\documentclass[a4j,11pt]{article}

\usepackage[dvipdfmx,hiresbb]{graphicx}
\usepackage[dvipdfmx]{color}
\usepackage{amsmath}
\usepackage{amsthm}
\usepackage{ascmac}
\usepackage{amssymb}
\usepackage{color}
\usepackage{lastpage}
\usepackage{natbib}

\usepackage{geometry}

\numberwithin{equation}{section}
\newtheorem{thm}{Theorem}
\newtheorem{cor}[thm]{Corollary}
\newtheorem{prop}[thm]{Proposition}
\newtheorem{rem}{Remark}
\newcommand{\ve}{\varepsilon}
\newcommand{\rk}{\mbox{\rm rank}}

\newcommand{\cov}{\mbox{\rm cov}}
\newcommand{\mm}{\mbox{MMSE}}
\newcommand{\GR}{\mbox{GR}}
\newcommand{\dif}{{\rm d_1}}
\newcommand{\diff}{{\rm d_2}}
\newcommand{\gme}{{GM}}
\newcommand{\ols}{{OLS}}

\title{
Covariance structure associated with an equality between two general ridge estimators
}
\author{Koji Tsukuda\footnote{Graduate School of Arts and Sciences, the University of Tokyo, 3-8-1 Komaba, Meguro-ku, Tokyo 153-8902, Japan. mail: ctsukuda@g.ecc.u-tokyo.ac.jp}; Hiroshi Kurata\footnote{Graduate School of Arts and Sciences, the University of Tokyo, 3-8-1 Komaba, Meguro-ku, Tokyo 153-8902, Japan. mail: kurata@waka.c.u-tokyo.ac.j}}

\begin{document}
\maketitle

\begin{abstract}
In a general linear model, this paper derives a necessary and sufficient condition under which two general ridge estimators coincide with each other.
The condition is given as a structure of the dispersion matrix of the error term.
Since the class of estimators considered here contains linear unbiased estimators such as the ordinary least squares estimator and the best linear unbiased estimator, our result can be viewed as a generalization of the well-known theorems on the equality between these two estimators, which have been fully studied in the literature.
Two related problems are also considered: equality between two residual sums of squares, and classification of dispersion matrices by a perturbation approach.
\\

\noindent MSC-2010. primary:62J05. secondary: 62F10, 62J07.  \\
key words and phrases: Best linear unbiased estimator; General linear model; Least squares estimator; Perturbation approach
\end{abstract}

\section{Introduction}\label{sec1}
In a general linear model, this paper derives a necessary and sufficient condition under which 
two general ridge estimators coincide with each other.
To state the problem more precisely, let us consider
\begin{equation}
y = X \beta + \ve, \quad {\sf E}[\ve]=0,  \quad {\sf E}[\ve \ve^\top ]=\sigma^2 \Omega, \label{GLM}
\end{equation}
where $y$ is an $n\times1$ vector, $X$ is an $n\times k$ matrix ($n>k$) satisfying $\rk (X)=k$, $\sigma^2$ is an unknown positive constant and $\Omega$ is a known positive definite matrix.
As is well-known,  the estimator of the form 
\[
\tilde\beta_\gme=(X^\top\Omega^{-1} X)^{-1} X^\top \Omega^{-1} y,
\]
which will be called the Gauss--Markov estimator in the sequel, is the best linear unbiased estimator of $\beta$, that is, 
it has the smallest covariance matrix  (in terms of positive semidefiniteness) among linear unbiased estimators.
This estimator is also optimal with respect to the following  quadratic risk functions:
\[
R(\tilde\beta,\beta)={\sf E}\left[ (\tilde\beta-\beta)^\top W(\tilde\beta-\beta)\right] ,
\]
where $W$ is an arbitrary positive semidefinite matrix. 
However, if we broaden the class of estimators to that of linear but not necessarily unbiased estimators, it is no longer optimal
and general ridge estimators play an essential role instead.
Here, a general ridge estimator is defined to be an estimator of the form
\begin{equation}
\hat\beta(\Psi,K) = (X^\top \Psi^{-1}X + K)^{-1} X^\top \Psi^{-1}y \quad {\rm with}\  \Psi \in \mathcal{S}^+(n) \ {\rm and} \ K\in \mathcal{S}^N(k)\label{GRE}
\end{equation}
(\cite{RefR76}), where $\mathcal{S}^+(m)$ and $\mathcal{S}^N(m)$ denote the sets of $m\times m$ positive definite and semidefinite matrices, respectively. 
As is proved by  \cite{RefR76} and \cite{RefM}, the general ridge estimators are linearly sufficient and linearly admissible, and conversely, any linearly sufficient and linearly admissible estimator belongs to the class of general ridge estimators.
Moreover, they are linearly complete.
For other properties of general ridge estimator, see, for example, \cite{RefAS}, \cite{RefG} and \cite{RefGM}.

On the other hand, it is also well-known that there are some cases in which two linear unbiased estimators coincide with each other.
Perhaps most important is the one in which the Gauss-Markov estimator $\tilde\beta_\gme$ is identically equal to the  ordinary least squares estimator  $\tilde\beta_\ols=(X^\top X)^{-1}X^\top y$, which does not depend on $\Omega$.
Conditions for the equality between  the two estimators have been studied by many authors so far (see, for example, \cite{RefBT}, Chapter 7 of \cite{RefKK}, \cite{RefPS} and \cite{RefZ}).
Among others, \cite{RefR67} proved that for a given $X$, the equality $\tilde\beta_\gme=\tilde\beta_\ols$ holds for all $y$ if and only if $\Omega$ is of the form
\begin{equation}
\Omega = X\Gamma X^\top  + Z\Delta Z^\top \quad {\rm for\ some\ }\Gamma\in \mathcal{S}^+(k)
{\rm \ and\ } \Delta \in \mathcal{S}^+(n-k),
 \label{RCS}
\end{equation}
where  $Z$ is an $n\times(n-k)$ matrix satisfying $X^\top Z=0$ and $\rk(Z)=n-k$, and will be fixed throughout. 
In this paper, we discuss an identical equality between two general ridge estimators.
More precisely, we derive a necessary and sufficient condition for $\Omega$ to guarantee that, for given $K_1, K_2\in \mathcal{S}^N(n)$, the equality 
\[
\hat\beta(\Omega,K_1)=\hat\beta(I,K_2)\ \ {\rm for\ any\ }y\in \Bbb{R}^n
\]
holds. 
This result, which will be presented in Section 2, can be regarded as an extension of \eqref{RCS}, since the class of general ridge estimators includes the Gauss--Markov and the ordinary least squares estimators.
Indeed we can readily see
\[ \tilde\beta_\gme=\hat\beta(\Omega,0) \quad {\rm and} \quad \tilde\beta_\ols=\hat\beta(I,0). \]
The class also contains the ordinary ridge estimators $\hat\beta(I,\lambda I)=(X^\top X+\lambda I)^{-1}X^\top y$ and $\hat\beta(\Omega,\lambda I)= (X^\top \Omega^{-1}X+\lambda I)^{-1}X^\top\Omega^{-1} y$ with $\lambda>0$ and shrinkage estimators of the form $\hat\beta(I, \rho X^\top  X)=\rho\tilde\beta_\ols$ and $\hat\beta(\Omega,\rho X^\top \Omega^{-1}X)=\rho\tilde\beta_\gme$ with $\rho>0$.

In Sections 3 and 4,  two related problems are considered:
First one is the problem of deriving 
a condition on $\Omega$ under which an identical equality between two generalized residual sums of squares holds.
To state it precisely, let
\begin{equation}
\GR (\Psi,K) = (y-X \hat\beta(\Psi,K))^\top  \Psi^{-1} (y-X \hat\beta(\Psi,K)) \quad \left( \Psi \in \mathcal{S}^+(n); \ K\in \mathcal{S}^N(k) \right).  \label{GRSS} 
\end{equation}
Then the ordinary residual sums of squares and its Gauss--Markov version are
given respectively by 
\[ \GR(I,0) = (y-X \hat\beta(I,0))^\top (y-X \hat\beta(I,0))\]
and
\[ \GR(\Omega,0) = (y-X \hat\beta(\Omega,0))^\top\Omega^{-1} (y-X \hat\beta(\Omega,0)). \]
In the literature, \cite{RefKa} derived a necessary and sufficient condition under which $\GR(\Omega,0)=\GR(I,0)$ in the context of estimation of $\sigma^2$. 
He also derived a condition for the two equalities $\hat\beta(\Omega,0)=\hat\beta(I,0)$ and $\GR(\Omega,0)=\GR(I,0)$ to hold  simultaneously.
The latter result was generalized by \cite{RefK}. 
See also \cite{RefG0}.
In this paper, we generalize their result by considering the case in which 
\[ \hat\beta(\Omega,K_1) = \hat\beta(I,K_2) \quad  {\rm and} \quad \GR(\Omega,K_1) = \GR(I,K_2) \]
hold for given $K_1$ and $K_2$. 
Needless to say, the above equalities do not generally hold.
Moreover, \cite{RefK} used 
\[ \rk \left( \cov(\hat\beta(I,0) - \hat\beta(\Omega,0)) \right)  =\rk \left( X^\top \Omega Z \right) \]
to measure the extent to which  $\Omega$ deviates from \eqref{RCS}.
In Section 4, we extend his result to the case including general ridge estimators.
 
As has been widely recognized, the simple ordinary ridge estimator $\hat\beta(I,\lambda I)$ shows better performance in practice than the ordinary least squares estimator $\hat\beta(I,0)$ when there exists a multicollinearity in the explanatory variables (\cite{RefHK}).
Moreover, some previous works such as \cite{RefFF} have reported that $\hat\beta(I,\lambda I)$ works well in many cases.
Hence, it is valuable to discuss the case $K\neq 0$ also from the practical viewpoint.

\section{Equality between two general ridge estimators}\label{sec2}

In this section, we derive a necessary and sufficient condition for the dispersion matrix $\Omega$ to guarantee 
an identical equality between two general ridge estimators.
We use the fact that the condition \eqref{RCS} is equivalent to $X^\top \Omega^{-1}Z=0$.

\begin{thm}\label{GRC}
For $K_1, K_2\in S^N (k)$, the equality $\hat\beta(\Omega , K_1) = \hat\beta(I , K_2)$ holds if and only if  the dispersion matrix $\Omega$ is of the form \eqref{RCS} with some $\Gamma \in \mathcal{S}^+(k)$ and $\Delta \in \mathcal{S}^+(n-k)$ satisfying
\begin{equation}
X^\top X\Gamma K_1 =  K_2 \label{t1j}. 
\end{equation}
\end{thm}

\begin{proof}
The equality $\hat\beta(\Omega , K_1) = \hat\beta(I , K_2)$ can be rewritten as 
\[ X^\top  \Omega^{-1} = (X^\top  \Omega^{-1} X + K_1)(X^\top X + K_2)^{-1} X^\top , \]
which is further equivalent to the following two equalities:
\begin{equation}
X^\top  \Omega^{-1} Z = 0  \label{p1g1}
\end{equation}
and
\begin{equation}
X^\top  \Omega^{-1} X = (X^\top  \Omega^{-1} X + K_1)(X^\top X + K_2)^{-1} X^\top  X, \label{p2g2}
\end{equation}
since $X^\top Z=0$ and the matrix $(X,Z)$ is nonsingular.
As is remarked in Section 1, the condition \eqref{p1g1} is equivalent to \eqref{RCS},
which can also be expressed as
\begin{equation} \Omega^{-1} = X (X^\top  X)^{-1} \Gamma^{-1} (X^\top  X)^{-1} X^\top  +  Z (Z^\top  Z)^{-1} \Delta^{-1} (Z^\top  Z)^{-1} Z^\top . \label{Oinv}
\end{equation}
Substituting it to \eqref{p2g2} shows that with \eqref{p1g1}, the condition \eqref{p2g2} is equivalent to
\begin{eqnarray*}
&& \Gamma^{-1} = (\Gamma^{-1} + K_1 )(X^\top X + K_2 )^{-1} X^\top  X \\
&\Leftrightarrow& I = (I + \Gamma K_1 )(X^\top X + K_2 )^{-1} X^\top  X \\
&\Leftrightarrow& (I + \Gamma K_1)^{-1} = (X^\top X + K_2)^{-1} X^\top  X \\
&\Leftrightarrow&  I + \Gamma K_1 = (X^\top  X)^{-1} (X^\top X + K_2) \\
&\Leftrightarrow&  \Gamma K_1 = (X^\top  X)^{-1} K_2.
\end{eqnarray*}
This completes the proof.\qed
\end{proof}

Using Theorem \ref{GRC} with $K_1=K_2=K$, we have the following corollary.

\begin{cor}\label{cor22}
For $K \in \mathcal{S}^N(k)$, the equality $\hat\beta(\Omega , K) = \hat\beta(I , K)$ holds if and only if 
$\Omega$ is of the form \eqref{RCS} with $\Gamma$ satisfying 
$X^\top  X \Gamma K = K.$
In particular, if $K $ is nonsingular, then
$\hat\beta(\Omega , K) = \hat\beta(I , K) \Leftrightarrow \Omega = X (X^\top  X) ^{-1} X^\top  + Z\Delta Z^\top$
for some $\Delta\in \mathcal{S}^+(n-k)$. 
\end{cor}

\begin{rem}\label{rm5}
Let $K=\lambda I\in \mathcal{S}^+(k)$.
Then Corollary \ref{cor22} implies that $\hat\beta(\Omega , \lambda I) = \hat\beta(I ,  \lambda I)$ is equivalent to
$\Omega = X(X^\top X)^{-1}X^\top  + Z\Delta Z^\top$
for some $\Delta \in \mathcal{S}^+(n-k)$.
\end{rem}

\begin{rem}\label{rm3}
Suppose that $\Omega$ satisfies \eqref{RCS}.
Let $K_1 = \rho X^\top  \Omega^{-1} X$
and $K_2 = \rho X^\top  X$ with $\rho>0$.
Then the two matrices satisfy the condition \eqref{t1j}.
In fact,  by \eqref{Oinv}, we see that $X^\top  \Omega^{-1} X=\Gamma^{-1}$ and hence
$K_1=\rho \Gamma^{-1}$, implying $X^\top X\Gamma K_1 =  X^\top X\Gamma (\rho \Gamma^{-1})=\rho X^\top X = K_2$.
Thus Theorem \ref{GRC} applies and hence the equality $\hat\beta(\Omega , \rho X^\top  \Omega^{-1} X) = \hat\beta(I , \rho X^\top  X)$ holds.
More specifically, the condition \eqref{RCS} is necessary and sufficient for $\hat\beta(\Omega , \rho X^\top  \Omega^{-1} X) = \hat\beta(I , \rho X^\top  X)$.
This conclusion itself is obvious from the forms of the shrinkage estimators.
\end{rem}

Next we clarify when there exists $\Gamma$ satisfying the condition \eqref{t1j} for given $K_1 , K_2 \in \mathcal{S}^N(k)$.
For this purpose, let
\begin{eqnarray*}
&& \bar{K}_i = (X^\top X)^{-1/2} K_i (X^\top X)^{-1/2}  \quad (i=1,2) ,\\
&& \bar{\Gamma} = (X^\top X)^{1/2} \Gamma (X^\top X)^{1/2}. 
\end{eqnarray*}
Needless to say, $\bar{K}_i$ and $\bar{\Gamma}$ have a one to one correspondence with $K_i$ and $\Gamma$, respectively.

\begin{prop}\label{propk}
There exists $\Gamma\in \mathcal{S}^+(k)$ satisfying \eqref{t1j} if and only if $\bar{K}_1$ and $\bar{K}_2$ satisfy
\begin{equation}
 \mathcal{R}(\bar{K}_1) = \mathcal{R}(\bar{K}_2)  \quad {\rm and} \quad \bar{K}_1 \bar{K}_2 = \bar{K}_2 \bar{K}_1, \label{condition}
\end{equation} 
where $\mathcal{R}(\bar{K}_i)$ denotes the range of $\bar{K}_i$. 
In this case, $\Gamma$ is of the form
\begin{eqnarray}
\Gamma = (X^\top X)^{-1/2} \left\{ \bar{K}_2 \bar{K}_1^+ + (I -  \bar{K}_1 \bar{K}_1^+) H (I -  \bar{K}_1 \bar{K}_1^+) \right\} (X^\top X)^{-1/2} \nonumber \\  {\rm for\ some}\ \ H \in \mathcal{S}^+(k), \label{GammaExpression}
\end{eqnarray}
where $\bar{K}_1^+$ denotes the Moore--Penrose inverse of $\bar{K}_1$.
\end{prop}

\begin{proof}
Suppose first that the condition \eqref{t1j} holds, which is equivalent to
\begin{equation}
\bar{\Gamma} \bar{K}_1 = \bar{K}_2.\label{condition2}
\end{equation}
Here, the two matrices in the left hand side commute, i.e., $\bar{K}_1\bar{\Gamma}=\bar{\Gamma}\bar{K}_1$,
since  $\bar{K}_2$ is symmetric.
Due to the nonsingularity of  $\bar{\Gamma} $, the matrices $\bar{K}_i$'s must satisfy $\mathcal{R}(\bar{K}_1) = \mathcal{R}(\bar{K}_2)$.
Hence, by letting  $\rk(\bar{K}_1) = \rk(\bar{K}_2) = r$, they can be commonly expressed as
\[ \bar{K}_i = V D_i V^\top \ (i=1,2)
\]
with $D_i \in \mathcal{S}^+(r)$ and $V$ a $k \times r$ matrix satisfying $V^\top V =I$.
Since $\bar {\Gamma} \bar{K}_1 = \bar{K}_1 \bar {\Gamma}$, we can write
$ \bar{\Gamma} = VFV^\top + WGW^\top $
for some $F\in \mathcal{S}^+(r)$, $G \in \mathcal{S}^+(k-r)$ and $W$ a $k \times (k-r)$ matrix satisfying $W^\top W =I$ and $V^\top W =0$.
Furthermore,  $D_1$ and  $F$ can be taken as diagonal matrices.
Hence the equality $ \bar{\Gamma} \bar{K}_1 = \bar{K}_2$ implies $V F D_1 V^\top = V D_2 V^\top ,$
which can be rewritten as $FD_1 =D_2$.
Therefore, $D_2$ must be also diagonal, which implies that $\bar{K}_1$ and $\bar{K}_2$ commute.
Thus we have \eqref{condition} and 
\begin{eqnarray}
 \bar{\Gamma} &=& V D_2 D_1^{-1} V^\top + W G W^\top \ \ {\rm for\ some}\ G \in \mathcal{S}^+(k-r), \nonumber \\
&=& V D_2 V^\top (VD_1^{-1} V^\top)+W W^\top H W W^\top \ \ {\rm for\ some}\ H \in \mathcal{S}^+(k),\nonumber 
\end{eqnarray}
where the last expression is equivalent to \eqref{GammaExpression}, since 
$\bar{K}_1^+=VD_1^+V^\top$ and $WW^\top=I-\bar{K}_1\bar{K}_1^+$.

Conversely, suppose that \eqref{condition} hold.
Then, by letting $\Gamma$ as in \eqref{GammaExpression}, we have
\begin{eqnarray*}
\bar{\Gamma}\bar{K}_1&=& \bar{K}_2 \bar{K}_1^+ \bar{K}_1+ (I -  \bar{K}_1 \bar{K}_1^+) H (I -  \bar{K}_1 \bar{K}_1^+)\bar{K}_1 \\
&=& \bar{K}_2 \bar{K}_1^+ \bar{K}_1 \\
&=& \bar{K}_2 \bar{K}_2^+ \bar{K}_2 \\ 
&=& \bar{K}_2,
\end{eqnarray*}
since $\mathcal{R}(\bar{K}_1)=\mathcal{R}(\bar{K}_2)$ implies $\bar{K}_1^+ \bar{K}_1=\bar{K}_2^+ \bar{K}_2$.
This shows the existence of $\Gamma$ that satisfies \eqref{condition2}, which is equivalent to \eqref{t1j}.
This completes the proof.\qed
\end{proof}

\begin{rem}
When $\Omega$ is unknown, it is often assumed that $\det(\Omega)=1$ to make the model identifiable (\citet{RefKa}).
In this case, in order that $\hat\beta(\Omega,K_1) = \hat\beta(I,K_2)$ holds, the matrices $\Gamma$ and $\Delta$ should satisfy
\[ \det(\Omega) 
= \det \left( (X,Z) \left(\begin{matrix}  \Gamma & 0 \\ 0 & \Delta \end{matrix} \right) \left(\begin{matrix}  X^\top  \\  Z^\top \end{matrix} \right) \right) 
= \left( \det((X , Z)) \right)^2 \det(\Gamma) \det(\Delta) =1 \]
as well as \eqref{t1j}.
In particular, when $K_1$ and $K_2$ are positive definite, the matrix $\Delta$ should satisfy
\[\det(\Delta) = \frac{\det(X^\top X) \det(K_1) }{ (\det((X,Z)))^{2} \det(K_2) }. \]
\end{rem}

\section{Equality between residual sums of squares}\label{sec3}
In this section, we discuss a condition under which the identical equality 
\[ \GR(\Omega,K_1)= \GR(I,K_2) \]
 holds in addition to  $\hat\beta(\Omega,K_1)=\hat\beta(I,K_2)$,
where the general residual sum of squares $ \GR(\Omega,K_1)$ is defined in \eqref{GRSS}.
To make notations simpler, let us denote
\begin{eqnarray}
&& A = (X^\top X)^{-1} \Gamma^{-1} (X^\top X)^{-1}, \nonumber \\ 
&& B = I - X(X^\top X + K_2)^{-1} X^\top,\label{AB}
\end{eqnarray}
where $A$ is positive definite and $B$ is positive semidefinite.

\begin{thm}\label{prop2}
For $K_1, K_2\in S^N (k)$, the two identical equalities 
\[ \hat\beta(\Omega,K_1)=\hat\beta(I,K_2) \quad {\rm and} \quad \GR(\Omega,K_1) = \GR(I,K_2) \]
simultaneously hold if and only if the following three conditions
\begin{eqnarray}
&& \Omega = X^\top \Gamma X + Z^\top  (Z^\top Z)^{-1} Z \label{t2j}  \\
&& X^\top X\Gamma K_1 =  K_2 \label{t2j0}  \\
&& X^\top B X  \{ A - (X^\top X)^{-1} \}  X^\top BX =0  \label{t2j1}
\end{eqnarray}
hold for some $\Gamma \in \mathcal{S}^+(k)$.
\end{thm}

\begin{proof}
From Theorem \ref{GRC}, $\hat\beta(\Omega,K_1)=\hat\beta(I,K_2)$ is equivalent to $\Omega = X \Gamma X^\top + Z \Delta Z^\top$ with $X^\top X\Gamma K_1 =  K_2$.
In this case, it holds that
\begin{equation}
 \Omega^{-1} = XAX^\top  + Z (Z^\top Z)^{-1} \Delta^{-1} (Z^\top Z)^{-1} Z^\top,  \label{Oinv}
\end{equation} 
and $y - X \hat\beta(\Omega,K_1)=y - X \hat\beta(I,K_2) $.
This implies that 
\[
\GR(\Omega,K_1) 
= (y-X\hat\beta(I,K_2))^\top \Omega^{-1}(y-X\hat\beta(I,K_2)) 
= y^\top B \Omega^{-1} B y
\]
with $B$ given in \eqref{AB},
and 
\[\GR(I,K_2) =  (y-X\hat\beta(I,K_2))^\top (y-X\hat\beta(I,K_2)) = y^\top B B y.\]
Thus the problem is to find a condition under which 
$ y^\top (B \Omega^{-1} B - B B ) y = 0 $
holds for arbitrary $y$, which is clearly equivalent to 
$B \Omega^{-1} B = B B$.
The quantities $B\Omega^{-1}B$ and $BB$ are calculated respectively as 
\begin{eqnarray*}
 B\Omega^{-1}B 
&=& X A X^\top + Z (Z^\top Z)^{-1} \Delta^{-1} (Z^\top Z)^{-1} Z^\top \\
&& - X(X^\top X+K_2)^{-1} X^\top X A X^\top  - XAX^\top X (X^\top X+K_2)^{-1} X^\top  \\
&& + X (X^\top X + K_2)^{-1} X^\top X A X^\top X (X^\top X +K_2)^{-1} X^\top
\end{eqnarray*}
and
\[ B B 
= I - 2 X(X^\top X +K_2)^{-1} X^\top + X(X^\top X +K_2)^{-1} X^\top X (X^\top X +K_2)^{-1} X^\top , \]
where \eqref{Oinv} is used.
Since 
\[ X^\top(B \Omega^{-1} B - BB )Z=0 \]
holds, the equality  
\[ B \Omega^{-1} B = B B\]
is equivalent to 
\begin{equation}
X^\top (B \Omega^{-1} B - BB) X =0
\label{p3g1}
\end{equation}
and
\begin{equation}
Z^\top (B \Omega^{-1} B - BB) Z =0.
\label{p3g2}
\end{equation}
The equality \eqref{p3g1} can be rewritten as
\begin{eqnarray*}
&& \Gamma^{-1} - X^\top X - X^\top X(X^\top X + K_2)^{-1} (\Gamma^{-1} - X^\top X) \\
&& - (\Gamma^{-1} - X^\top X) (X^\top X +K_2)^{-1} X^\top X  \\
&& + X^\top X (X^\top X +K_2)^{-1}  (\Gamma^{-1} - X^\top X) (X^\top X +K_2)^{-1} X^\top X  =0 \\
&\Leftrightarrow& \{ I - X^\top X(X^\top X+K_2)^{-1} \} (\Gamma^{-1} - X^\top X) \{ I -(X^\top X+K_2)^{-1} X^\top X \}   =0 \\
&\Leftrightarrow& X^\top \{ I - X(X^\top X+K_2)^{-1} X^\top \} X (X^\top X)^{-1} (\Gamma^{-1} - X^\top X) \\ && (X^\top X)^{-1} X^\top \{ I -(X^\top X+K_2)^{-1} X^\top  \} X =0 \\
&\Leftrightarrow& X^\top B X  \{ A - (X^\top X)^{-1} \}  X^\top BX =0 
\end{eqnarray*}
and \eqref{p3g2} is the same as
\[ \Delta^{-1} - Z^\top Z =0 \Leftrightarrow \Delta = (Z^\top Z)^{-1} . \]
This completes the proof. \qed
\end{proof}

\begin{rem}
The above theorem can be viewed as an extension of 
\cite{RefKa} (Corollary), in which it is shown that \eqref{t2j} is a necessary and sufficient condition under which $\hat\beta(\Omega,0)=\hat\beta(I,0)$ and $\GR(\Omega,0) = \GR(I,0)$ simulataneously hold.
In fact, in Theorem \ref{prop2}, let $K_1=K_2=0$.
Then the conditions \eqref{t2j0} and \eqref{t2j1} vanish, since they hold for all $\Gamma\in \mathcal{S}^+(k)$.
Hence, the conditions in the above theorem reduces to \eqref{t2j}.
\end{rem}

\begin{cor}\label{cor33}
Let $K \in \mathcal{S}^+(k)$.
The two equalities $\hat\beta(\Omega , K) = \hat\beta(I , K)$ and $\GR(\Omega,K) = \GR(I,K)$ simultaneously hold if and only if $\Omega = I$.
\end{cor}

\begin{proof}
Letting $K_1=K_2=K$, we will use Theorem \ref{prop2}.
From \eqref{t2j0}, $\Gamma=(X^\top X)^{-1}$.
Since $A=(X^\top X)^{-1}$ which yields \eqref{t2j1}, Theorem \ref{prop2} implies that both $\hat\beta(\Omega,\lambda I)=\hat\beta(I,\lambda I)$ and $\GR(\Omega,\lambda I) = \GR(I,\lambda I)$ hold if and only if
\[ \Omega =X  (X^\top  X)^{-1} X^\top + Z  (Z^\top Z)^{-1} Z^\top = I . \]
This completes the proof.\qed
\end{proof}

\begin{rem}
Let $K =\lambda I \in \mathcal{S}^+(k)$.
Then Corollary \ref{cor33} implies that the two equalities $\hat\beta(\Omega , \lambda I) = \hat\beta(I , \lambda I)$ and $\GR(\Omega, \lambda I) = \GR(I, \lambda I)$ simultaneously hold if and only if $\Omega = I$.
\end{rem}

\section{Classification criterion of dispersion matrices}\label{sec4}

As is observed in the previous sections, the condition in Theorem \ref{GRC} on $\Omega$  rarely holds, and hence
the estimators $\hat\beta(\Omega,K_1)$ and $\hat\beta(I,K_2)$ do not coincide in most cases.
In the context of comparing the Gauss-Markov and the ordinary least squares estimators,
\cite{RefK} used 
\begin{equation}
 \rk \left( {\cov}(\hat\beta(I,0) - \hat\beta(\Omega,0)) \right) 
= \rk(X^\top \Omega Z) =\rk(X^\top \Omega^{-1} Z)  
\label{KC} 
\end{equation}
as a criterion to measure the difference between them.
The rank ranges from $0$ to $\min(k, n-k)$ and takes zero if and only if $\Omega $ is of the form \eqref{RCS}.
Hence this criterion can also be regarded as a measure of the extent to which the structure of $\Omega$ deviates from \eqref{RCS},
or equivalently, a criterion to classify $\Omega$.
This section is devoted to deriving a generalization of his result to the case including general ridge estimators.

Since the quantity \eqref{KC} is the same as
\[ \rk \left( {\sf E} \left[ (\hat\beta(I,0) - \hat\beta(\Omega,0)) (\hat\beta(I,0) - \hat\beta(\Omega,0))^\top  \right] \right) ,\]
it is natural to use the rank of $L^2$ difference matrix 
\begin{equation} \label{prc1}
 \dif(\Omega,K_1,K_2)= {\sf E} \left[ (\hat\beta(I,K_2) - \hat\beta(\Omega,K_1)) (\hat\beta(I,K_2) - \hat\beta(\Omega,K_1))^\top \right]
\end{equation}
as a measure that is applicable to general ridge estimators.
Since we have
\begin{eqnarray*}
&& {\cov} \left( \hat\beta(I,0) - \hat\beta(\Omega,0) \right) \\
&=&  {\cov} \left( \left\{ (X^\top X)^{-1} X^\top - (X^\top \Omega^{-1} X)^{-1} X^\top \Omega^{-1} \right\} y \right) \\
&=& \sigma^2 \left\{ (X^\top X)^{-1} X^\top \Omega X (X^\top X)^{-1} - ( X^\top \Omega^{-1} X )^{-1} \right\} \\
&=&  {\cov}\left( \hat\beta(I,0) \right) - \cov \left( \hat\beta(\Omega,0) \right) \\
&=& \mm\left(\hat\beta(I,0)\right) -  \mm \left( \hat\beta(\Omega,0) \right),
\end{eqnarray*}
where 
\[ \mm \left( \hat\beta(\Psi,K) \right) ={ \sf E} \left[ (\hat\beta(\Psi,K)-\beta)(\hat\beta(\Psi,K)-\beta)^\top \right] \]
is the mean square error matrix (see, for example, (3.1) of \cite{RefG}), it is also natural to use
the rank of 
\begin{equation} \label{prc2}
\diff(\Omega,K_1,K_2)=
  \mm\left(\hat\beta(I,K_2)\right) -  \mm\left(\hat\beta(\Omega,K_1)\right).  
\end{equation}
We adopt the above two quantities in the sequel.

However, since it is in general not easy to analyze them unless $K_i=0$,
we limit our consideration to  the case in which both $K_1$ and $K_2$ are small.
More precisely, we fix $L_1$, $L_2$ in $\mathcal{S}^N(k)$ and use the perturbation approach by
letting $K_1 =\epsilon L_1$, $K_2 =\epsilon L_2$ with a small positive constant $\epsilon$.
Note that $\Omega\in \mathcal{S}^+(n)$ can be expressed as
\begin{equation}
\Omega= X\Gamma X^\top  + Z\Delta Z^\top  + X \Xi Z^\top  + Z \Xi^\top  X^\top
\label{Oexp}
\end{equation}
for some $\Gamma\in \mathcal{S}^+(k)$, $\Delta\in\mathcal{S}^+(n-k)$
and $\Xi: k\times (n-k)$.

\begin{thm} \label{prop4}
Fix $L_1, L_2 \in \mathcal{S}^N(k)$ and $\Omega\in \mathcal{S}^+(n)$, and write $\Omega$ as in \eqref{Oexp}.
Consider general ridge estimators $\hat\beta(\Omega,K_1)$ and $\hat\beta(I,K_2)$
with 
\[ K_1 =\epsilon L_1, \quad K_2 =\epsilon L_2\]
and $\epsilon$ a positive constant satisfying 
\[ \epsilon  ( \max \{ \| (\Gamma - \Xi \Delta^{-1} \Xi^\top ) L_1 \| , \| (X^\top  X)^{-1} L_2 \| \} )< 1, \] where $\|\cdot \|$ denotes a matrix norm.
Then the quantities $\dif$ in \eqref{prc1} and $\diff$ in \eqref{prc2} are evaluated as
\begin{eqnarray}
&& \dif(\Omega,\epsilon L_1,\epsilon L_2)   \nonumber \\
&=& \sigma^2 \Xi \Delta^{-1} \Xi^\top   \nonumber \\
&& - \sigma^2 \left\{  \Xi \Delta^{-1} \Xi^\top L_2 (X^\top X)^{-1} +  (X^\top X)^{-1} L_2 \Xi \Delta^{-1} \Xi^\top  \right\} \epsilon  + O(\epsilon^2) \label{p4r1}
\end{eqnarray}
and 
\begin{eqnarray}
&& \diff(\Omega,\epsilon L_1,\epsilon L_2)  \nonumber \\
&=& \sigma^2 \Xi \Delta^{-1} \Xi^\top  \nonumber
\\ &&  - \sigma^2 \left\{ (X^\top X)^{-1}L_2 \Gamma +\Gamma L_2 (X^\top X)^{-1} - 2(\Gamma - \Xi \Delta^{-1} \Xi^\top )L_1(\Gamma - \Xi \Delta^{-1} \Xi^\top) \right\} \epsilon \nonumber \\  && + O(\epsilon^2), \label{p4r2}
\end{eqnarray}
respectively, as $\epsilon \downarrow 0$.
\end{thm}

\begin{proof}
First we prove \eqref{p4r1}.
Clearly, $\dif(\Omega,K_1,K_2)$ is equal to
\begin{eqnarray}
&& \cov \left( \hat\beta(I,K_2) - \hat\beta(\Omega,K_1) \right) \nonumber \\ && + {\sf E} \left[ \hat\beta(I,K_2) - \hat\beta(\Omega,K_1) \right]  {\sf E} \left[ (  \hat\beta(I,K_2) - \hat\beta(\Omega,K_1) )^\top \right] . \label{p4g1}
\end{eqnarray}
As for the first term of \eqref{p4g1}, we have
\begin{eqnarray*}
&& \sigma^{-2} \cov \left( \hat\beta(I,K_2)  - \hat\beta(\Omega,K_1) \right) \\
&=&  \sigma^{-2} \cov \left(  \{ (X^\top X+K_2)^{-1} X^\top - (X^\top \Omega^{-1} X +K_1)^{-1} X^\top \Omega^{-1} \} y \right)  \\
&=& \left\{ (X^\top X+K_2)^{-1} X^\top - (X^\top \Omega^{-1} X +K_1)^{-1} X^\top \Omega^{-1}  \right\} \Omega \\
&& \quad  \left\{ (X^\top X+K_2)^{-1} X^\top - (X^\top \Omega^{-1} X +K_1)^{-1} X^\top \Omega^{-1} \right\}^\top \\
&=& (X^\top X+K_2)^{-1} X^\top \Omega X (X^\top X+K_2)^{-1} \\
&& -  (X^\top \Omega^{-1} X+K_1)^{-1} X^\top X  (X^\top X+K_2)^{-1} \\
&& -  (X^\top X+K_2)^{-1}  X^\top X  (X^\top \Omega^{-1} X+K_1)^{-1} \\
&& +  (X^\top \Omega^{-1} X+K_1)^{-1} X^\top \Omega^{-1} X  (X^\top \Omega^{-1} X+K_1)^{-1} \\
&=& \{ I+(X^\top X)^{-1}K_2 \}^{-1} (X^\top X)^{-1} X^\top \Omega X (X^\top X)^{-1} \{ I+K_2 (X^\top X)^{-1} \}^{-1}\\
&& -  \{I + (X^\top \Omega^{-1} X)^{-1} K_1\}^{-1} (X^\top \Omega^{-1} X)^{-1}  \{I + K_2 (X^\top X)^{-1}\}^{-1} \\
&& -  \{I + (X^\top X)^{-1} K_2\}^{-1} (X^\top \Omega^{-1} X)^{-1} \{I + K_1 (X^\top \Omega^{-1} X)^{-1}\}^{-1} \\
&& +  \{I + (X^\top \Omega^{-1} X)^{-1} K_1\}^{-1} (X^\top \Omega^{-1} X)^{-1}  \{I + K_1 (X^\top \Omega^{-1} X)^{-1} \}^{-1}.
\end{eqnarray*}
The four terms in the right-hand side are further calculated as
\begin{eqnarray}
&& \{I+(X^\top X)^{-1} K_2\}^{-1} \Gamma \{I+K_2 (X^\top X)^{-1}\}^{-1}, \label{p41} \\
&& \{I +(\Gamma - \Xi \Delta^{-1} \Xi^\top) K_1\}^{-1} (\Gamma - \Xi \Delta^{-1} \Xi^\top )  \{I + K_2 (X^\top X)^{-1} \}^{-1},  \label{p42} \\
&& \{ I +(X^\top X)^{-1} K_2 \}^{-1} (\Gamma - \Xi \Delta^{-1} \Xi^\top ) \{ I + K_1 (\Gamma - \Xi \Delta^{-1} \Xi^\top ) \}^{-1},  \label{p43} \\
&&  \{ I +(\Gamma - \Xi \Delta^{-1} \Xi^\top ) K_1 \}^{-1} (\Gamma - \Xi \Delta^{-1} \Xi^\top )  \{ I + K_1 (\Gamma - \Xi \Delta^{-1} \Xi^\top ) \}^{-1}, \nonumber \\ \label{p44} 
\end{eqnarray}
respectively.
As for the second term of \eqref{p4g1}, we obtain
\begin{eqnarray*}
{\sf E} \left[ \hat\beta(I,K_2) \right]
= (X^\top X+K_2)^{-1} X^\top X \beta 
= \{ I+(X^\top X)^{-1}K_2 \}^{-1} \beta
\end{eqnarray*}
and
\begin{eqnarray*}
 {\sf E} \left[ \hat\beta(\Omega,K_1) \right] 
&=& (X^\top \Omega^{-1}X+K_1)^{-1} X^\top  \Omega^{-1} X\beta \\
&=& \{ (\Gamma - \Xi \Delta^{-1} \Xi^\top )^{-1} +K_1 \}^{-1} (\Gamma - \Xi \Delta^{-1} \Xi^\top )^{-1} \beta \\
&=& \{ I +  (\Gamma - \Xi \Delta^{-1} \Xi^\top ) K_1 \}^{-1} \beta .
\end{eqnarray*}
From the definition of matrix functions, equations \eqref{p41}--\eqref{p44} are evaluated as
\begin{eqnarray*}
&& \Gamma - \{ (X^\top X)^{-1} L_2 \Gamma + \Gamma L_2 (X^\top X)^{-1} \} \epsilon + O(\epsilon^2), \\
&& \Gamma - \Xi \Delta^{-1}\Xi^\top - \{ (\Gamma - \Xi \Delta^{-1}\Xi^\top)L_1(\Gamma - \Xi \Delta^{-1}\Xi^\top)  \\ && \quad + (\Gamma - \Xi \Delta^{-1}\Xi^\top) L_2 (X^\top X)^{-1} \} \epsilon + O(\epsilon^2), \\
&& \Gamma - \Xi \Delta^{-1}\Xi^\top - \{ (X^\top X)^{-1} L_2 (\Gamma - \Xi \Delta^{-1}\Xi^\top) \\ && \quad+ (\Gamma - \Xi \Delta^{-1}\Xi^\top)L_1(\Gamma - \Xi \Delta^{-1}\Xi^\top) \} \epsilon + O(\epsilon^2), \\
&& \Gamma - \Xi \Delta^{-1}\Xi^\top - 2 (\Gamma - \Xi \Delta^{-1}\Xi^\top)L_1(\Gamma - \Xi \Delta^{-1}\Xi^\top) \epsilon + O(\epsilon^2) ,
\end{eqnarray*}
respectively.
On the other hand, from the definition of matrix functions again, it holds that
${\sf E}[ \hat\beta(I,K_2) ] =\beta +O(\epsilon)$
and
${\sf E} [ \hat\beta(\Omega,K_1) ] =\beta +O(\epsilon)$,
which implies that
\[ {\sf E} \left[ \hat\beta(I,K_2) - \hat\beta(\Omega,K_1) \right] {\sf E} \left[ ( \hat\beta(I,K_2) - \hat\beta(\Omega,K_1) )^\top \right]  =  O(\epsilon^2). \]
Thus we have
\begin{eqnarray*} 
&& {\sf E} \left[ (\hat\beta(I,K_2) - \hat\beta(\Omega,K_1)) (\hat\beta(I,K_2) - \hat\beta(\Omega,K_1))^\top \right] \\
&=& \cov \left(  \hat\beta(I,K_2) - \hat\beta(\Omega,K_1)  \right)  + O(\epsilon^2) \\
&=&\sigma^2 \Xi \Delta^{-1} \Xi^\top - \sigma^2 \left\{  \Xi \Delta^{-1} \Xi^\top L_2 (X^\top X)^{-1} +  (X^\top X)^{-1} L_2 \Xi \Delta^{-1} \Xi^\top  \right\} \epsilon 
 + O(\epsilon^2).
\end{eqnarray*}

Next we prove \eqref{p4r2}.
Clearly, $\diff(\Omega,K_1,K_2)$ is equal to
\begin{eqnarray}
&&  \cov \left( \hat\beta(I,K_2) \right) - \cov \left( \hat\beta(\Omega,K_1) \right)
+ {\sf E} \left[ \hat\beta(I,K_2) - \beta \right]  {\sf  E} \left[ ( \hat\beta(I,K_2) - \beta )^\top  \right] \nonumber \\
&& - {\sf  E} \left[ \hat\beta(\Omega,K_1) - \beta \right] { \sf E} \left[ ( \hat\beta(\Omega,K_1) - \beta )^\top \right] . \label{p4g2}
\end{eqnarray}
As for the first and second terms of \eqref{p4g2}, it holds that
\begin{eqnarray*}
\cov \left( \hat\beta(I,K_2) \right)   
&=& \sigma^2 (X^\top X+K_2)^{-1} X^\top X \Gamma X^\top X (X^\top X +K_2)^{-1} \\
&=& \sigma^2  \{I + (X^\top X)^{-1} K_2 \}^{-1} \Gamma  \{ I +  K_2 (X^\top X)^{-1} \}^{-1}
\end{eqnarray*}
and 
\begin{eqnarray*}
&& \cov \left( \hat\beta(\Omega,K_1) \right) \\
&=& \sigma^2  \{ (\Gamma - \Xi \Delta^{-1} \Xi^\top)^{-1} +K_1 \}^{-1} (\Gamma - \Xi \Delta^{-1} \Xi^\top)^{-1}  \{ (\Gamma - \Xi \Delta^{-1} \Xi^\top)^{-1} +K_1 \}^{-1} \\
&=& \sigma^2  \{ I + (\Gamma - \Xi \Delta^{-1} \Xi^\top) K_1 \}^{-1} \{ I + (\Gamma - \Xi \Delta^{-1} \Xi^\top) K_1 \}^{-1}  (\Gamma - \Xi \Delta^{-1} \Xi^\top).
\end{eqnarray*}
Moreover, recall that ${\sf E} [\hat\beta(I,K_2) ] = \{ I+(X^\top X)^{-1}K_2 \}^{-1} \beta$ and ${\sf E} [ \hat\beta(\Omega,K_1) ) = \{ I +  (\Gamma - \Xi \Delta^{-1} \Xi^\top ) K_1 \}^{-1} \beta$.
From the definition of matrix functions, it follows that
\begin{eqnarray*}
&& {\sf E} \left[ \hat\beta(I,K_2) - \beta \right] {\sf  E} \left[ ( \hat\beta(I,K_2) - \beta )^\top \right] \\ && - {\sf  E} \left[ \hat\beta(\Omega,K_1) - \beta \right] {\sf  E} \left[( \hat\beta(\Omega,K_1) - \beta )^\top \right]   =  O(\epsilon^2) , 
\end{eqnarray*}
and hence 
\[ \diff(\Omega,K_1,K_2) = \cov(\hat\beta(I,K_2)) - \cov(\hat\beta(\Omega,K_1))  + O(\epsilon^2), \]
which is equal to 
\begin{eqnarray*}
&& \sigma^2 \Xi \Delta^{-1} \Xi^\top - \sigma^2 \left\{  (X^\top X)^{-1}L_2 \Gamma + \Gamma L_2 (X^\top X)^{-1} \right. \\ && \left. - 2(\Gamma - \Xi \Delta^{-1} \Xi^\top) L_1 (\Gamma - \Xi \Delta^{-1} \Xi^\top) \right\} \epsilon + O(\epsilon^2). 
\end{eqnarray*}
This completes the proof.\qed
\end{proof}

\begin{rem}
Since 
\[ \Gamma = (X^\top X)^{-1} X^\top \Omega X (X^\top X)^{-1} \quad {\rm and} \quad \Gamma - \Xi \Delta^{-1} \Xi^\top = (X^\top \Omega^{-1} X)^{-1},\]
the quantity 
\[ (X^\top X)^{-1}L_2 \Gamma + \Gamma L_2 (X^\top X)^{-1} - 2(\Gamma - \Xi \Delta^{-1} \Xi^\top)L_1(\Gamma - \Xi \Delta^{-1} \Xi^\top)\]
can be written in original notation as
\begin{eqnarray}
&& (X^\top X)^{-1} \{ L_2  (X^\top X)^{-1} X^\top \Omega X + X^\top \Omega X  (X^\top X)^{-1} L_2 \} (X^\top X)^{-1} \nonumber \\ &&  - 2 (X^\top \Omega^{-1} X)^{-1} L_1 (X^\top \Omega^{-1} X)^{-1}. \label{TKC2} 
\end{eqnarray}
If $\Omega$ is of the form \eqref{RCS},  then  \eqref{TKC2} is simplified as 
\begin{eqnarray*}
&& \{ (X^\top X)^{-1}L_2 - (X^\top \Omega^{-1} X)^{-1} L_1 \}(X^\top \Omega^{-1} X)^{-1} \\ && + (X^\top \Omega^{-1} X)^{-1} \{ L_2 (X^\top X)^{-1} - L_1 (X^\top \Omega^{-1} X)^{-1}  \}. 
\end{eqnarray*}
In particular, when $\Omega = I$, the above quantity is further reduced to 
\[ 2  (X^\top X)^{-1} (L_2 - L_1)  (X^\top X)^{-1} . \]

If we consider estimators such that $K_1=\rho X^\top \Omega^{-1} X$ and  $K_2 =\rho X^\top X$, then the matrices
$L_1$ and $L_2$ are given by $L_1 =X^\top \Omega^{-1} X$ and $L_2 = X^\top X$,
where the constant $\rho$ is absorbed into $\epsilon$.
In this case, the quantity \eqref{TKC2} takes the form 
\[ 2\rho (X^\top X)^{-1} X^\top \Omega X (X^\top X)^{-1} -2\rho (X^\top \Omega^{-1} X)^{-1}  .\]
\end{rem}

From Theorem \ref{prop4}, when $\epsilon$ is small, the major part of the deviation of the simple estimator $\hat\beta(I,K_2)$ from the good estimator $\hat\beta(\Omega,K_1)$ is characterized by the first term 
\[\sigma^2 \Xi \Delta^{-1} \Xi^\top. \]
Moreover, when $\Xi=0$ (i.e., $\Omega$ is of the form \eqref{RCS}), the first term vanishes and the second term becomes 
\[ \sigma^2 \left\{ (X^\top X)^{-1}L_2\Gamma + \Gamma L_2 (X^\top X)^{-1} - 2 \Gamma L_1 \Gamma \right\} , \]
which is the coefficient of $\epsilon$.
If further $K_2=X^\top X \Gamma K_1$,  then $L_2=X^\top X \Gamma L_1=L_1\Gamma X^\top X$ holds and hence $ (X^\top X)^{-1}L_2\Gamma + \Gamma L_2 (X^\top X)^{-1} - 2 \Gamma L_1 \Gamma$ also vanishes, implying
$\dif(\Omega,K_1,K_2)=O(\epsilon^2)$ and $\diff(\Omega,K_1,K_2)=O(\epsilon^2)$.

From this observation, if both $\det(K_1)$ and $\det(K_2)$ are small, then the criterion proposed by \cite{RefK} still works even in the case including general ridge estimators. 
As its extension, based on \eqref{p4r2}, we propose the following two step criterion for classification of dispersion matrices:
\begin{enumerate}
\item Classify according to 
\[v_1 = \rk(\sigma^2 \Xi \Delta^{-1} \Xi^\top)=\rk(\Xi) = \rk(X^\top \Omega Z) ; \]
\item Make a finer classification by using 
\begin{eqnarray*}
v_2 &=& \rk(  (X^\top X)^{-1} K_2 \Gamma + \Gamma K_2 (X^\top X)^{-1} \\ && \quad - 2(\Gamma - \Xi \Delta^{-1} \Xi^\top) K_1 (\Gamma - \Xi \Delta^{-1} \Xi^\top) ) .
\end{eqnarray*}
\end{enumerate}

\begin{rem}
When $K_1=K_2=0$, then $v_2=0$ and our criterion reduces to that of \cite{RefK}.
\end{rem}

\begin{rem}
Let $\Omega =I$.
Then $v_1=0$ and $v_2 = \rk(K_2 - K_1)$.
As is seen in this example, even if we consider the same dispersion matrix, classification may vary according to the choice of $K_1$ and $K_2$.
Besides, when $K_1=\lambda_1 I$, $K_2=\lambda_2 I$ and $\lambda_1\neq\lambda_2 > 0$, then $v_2=k$.
\end{rem}

\begin{rem}
If $\Omega$ is of the form \eqref{RCS}, that is $\Xi = 0$, then 
\[ v_2 = \rk\left( X^\top \Omega^{-1} X(X^\top X)^{-1} K_2 + K_2 (X^\top X)^{-1}X^\top \Omega^{-1} X - 2 K_1   \right). \]
\end{rem}

\begin{rem}
When $K_1 =\rho X^\top \Omega^{-1} X$ and $K_2 =\rho X^\top X$ with $\rho > 0$, 
\[ v_2 = \rk\left( (X^\top X)^{-1} X^\top \Omega X (X^\top X)^{-1} - (X^\top \Omega^{-1} X)^{-1} \right). \]
\end{rem}

\section*{Acknowledgments}
The portion of the second author's work was supported by Japan Society for the Promotion of Science KAKENHI Grant Number JP26330035.

This is a pre-print of an article published in Statistical Papers.
The final authenticated version is available online at: https://doi.org/10.1007/s00362-017-0975-8.

\end{document}